\documentclass[11pt]{amsart}

\usepackage{url}
\usepackage{epsfig,color}
\usepackage{amssymb,amsrefs}
\usepackage{epstopdf}

\newtheorem{theorem}{Theorem}
\newtheorem{lemma}[theorem]{Lemma}
\newtheorem{remark}{Remark}
\newtheorem{cor}{Corollary}

\newcommand{\ZZ}{{\mathbb Z}}
\newcommand{\FF}{{\mathbb F}}
\newcommand{\QQ}{{\mathbb Q}}
\newcommand{\fq}{{\FF_q}}

\newcommand{\ftwo}{{\FF_2}}

\newcommand{\fthree}{{\FF_3}}

\newcommand{\Disc}{\operatorname{Disc}}
\newcommand{\Res}{\operatorname{Res}}

\newtheorem{conj}[theorem]{Conjecture}

\newcommand{\fp}{\FF_p}

\begin{document}
	
	\title{A note on the stability of trinomials over finite fields}

	\author[O. Ahmadi]{Omran Ahmadi}
	\address{Institute for Research in Fundamental Sciences, Iran}
	\email{oahmadid@ipm.ir}
	
	\author[K. Monsef-Shokri]{Khosro Monsef-Shokri}
	\address{School of Mathematical Science, Shahid Beheshti University,
		Iran}
	\email{k\_shokri@sbu.ac.ir}
	
	\date{August 16, 2018.}

\begin{abstract}
A polynomial $f(x)$ over a field $K$ is called stable if all of its iterates are irreducible over $K$.	In this paper we study the stability of trinomials over finite fields. Specially, we show that if $f(x)$ is a trinomial of even degree over the binary field $\ftwo$, then $f(x)$ is not stable. We prove a similar result for some families of monic trinomials over finite fields of odd characteristic. These results are obtained towards the resolution of a conjecture on the instability of polynomials over finite fields whose degrees are divisible by the characteristic of the underlying field. 
\end{abstract}

\maketitle
\let\thefootnote\relax\footnotetext{ Mathematical Classification Subject: 11T06,12E20\\
	Keywords: Polynomials, Iterations, Stability, Finite fields\nonumber}
Let $K$ be a field, and let $K[x]$ denote the polynomial ring over $K$. For a polynomial $f(x)\in K[x]$, its $n$-th iterate for $n\ge 0$ is defined inductively  by the following equations
\[
f^{(0)}(x)=x, f^{(n)}(x)=f(f^{(n-1)}(x).
\]
A polynomial $f(x)\in K[x]$ is called stable if $f^{(n)}(x)$ is irreducible over $K$ for every $n\ge 1$. Studying the stability of polynomials had attracted the attention of many researchers (see for example~\cite{AOLS,Nidal,Ayad,DNOS,Goksel,Jones-Boston,Odoni}). In this paper we are interested in the stability of polynomials over finite fields. Let $\fp$ denotes the finite field with $p$ elements where $p$ is a prime number. This paper has been inspired by a question raised by Domingo Gomez-Perez~\cite{Domingo} who based on some computations asked whether it is true that if the degree of the polynomial $f(x)\in\fp[x]$ is divisible by the prime number $p$, then $p+1$-th iterate of $f(x)$, i.e., $f^{(p+1)}(x)$, is not an irreducible polynomial over $\fp$. Our computations show that the answer to his question is negative since if $f(x)=x^{10} + x^9 + x^6 + x^5 + x^4 + x^3 + 1$, then $f(x),ff(x)$ and $fff(x)$ are irreducible and $f^{(4)}(x)$ is not irreducible over $\ftwo$. Though the answer to Gomez-Perez's question turned out to be negative, based on our computations we make the following conjecture.

\begin{conj}\label{main-conj}
  If the degree of the polynomial $f(x)\in\fp[x]$ is divisible by the prime number $p$,	then $f(x)$ is not stable over $\fp$.
\end{conj}
 
The conjecture is trivial when $f(x)$ is a binomial since if $f(x)=x^{pl}+a$, then $f(x)=(x^l+a)^p$. 
In his pioneering work studying the stability of polynomials ~\cite{Odoni}, Odoni showed that the additive polynomial $f(x)=x^p-x-1$ is not stable over the finite field $\fp$.  In~\cite{AOLS}, it was shown that there is no stable quadratic polynomial over the fields of characteristic two. In~\cite{DNOS}, it was shown that certain cubic polynomials which are trinomial, i.e, polynomials with three nonzero terms are not stable. All these cases can be considered as special cases for which Conjecture~\ref{main-conj} is true. Considering these cases it seems that the next natural step towards the proof the conjecture would be to confirm it for trinomials. In this paper, we study stability of monic trinomials over finite fields and confirm Conjecture~\ref{main-conj} for all trinomials over the finite field $\ftwo$ and some families of trinomials over finite fields of odd characteristic. We as well present some results dealing with polynomials of higher weight.

This paper is organized as follows. In Section~\ref{Prem}, we gather some preliminary results. In Section 2, we prove our main results about trinomials over finite fields. In Section 3, we present some results on the stability of polynomials of higher weight. Finally, Section 4 contains our concluding remarks.

\section{Preliminaries}\label{Prem}
In this section we prove and gather some results which will be used in the rest
of the paper to prove the main results of the paper.

\subsection{Newton Identities} Let $L$ be a field, and let $x_1,x_2,\ldots,x_m\in L$. 
 We denote by $e_k(x_1,x_2,\ldots,x_m)$ and $p_k(x_1,x_2,\ldots,x_m)$ the $k$-th elementary symmetric polynomial and the $k$-th power sum in $x_1,x_2,\ldots,x_m$,  respectively, i.e, 
\[
e_k(x_1,x_2,\ldots,x_m)=\sum_{1\le i_1<i_2<\ldots<i_k\le m}x_{i_1}x_{i_2}\ldots x_{i_k},
\] 
\[
p_k(x_1,x_2,\ldots,x_m)=\sum_{i=1}^{m}x_i^k.
\]
If $L$ is of characteristic zero and we let $p_0(x_1,x_2,\ldots,x_m)=m$, then for $k\ge 1$ and $l=\min(k,m)$ we have
\[
p_k-p_{k-1}e_1+p_{k-2}e_2+\cdots+(-1)^{l-1}p_{k-l+1}e_{l-1}+(-1)^l\frac{l}{m}p_{k-l}e_l=0.
\]

\subsection{Polynomial transformation of irreducible polynomials}\

The following lemma is known as Capelli's lemma and can be found in ~\cite{Cohen-1} 
too.
\begin{lemma}\label{cohen}
Let $f(x)$ be a degree $n$ irreducible polynomial over $\mathbb{F}_q$, and
let $g(x), h(x)\in \mathbb{F}_q[x]$. Then $p(x)=h(x)^nf(g(x)/h(x))$ is irreducible over 
$\mathbb{F}_q$ if and only if for some root $\alpha$ of $f(x)$ in 
$\mathbb{F}_{q^n}$, $g(x)-\alpha h(x)$ is an irreducible polynomial over 
$\mathbb{F}_{q^n}$.
\end{lemma}

\subsection{Parity of the number of the irreducible factors of a polynomial over
finite fields}{\bf{Be careful with the characteristic}}
We begin this section by recalling the {\it Discriminant} and the
{\it Resultant} of polynomials over a field. For a more detailed treatment
 see \cite[Ch. 1, pp. 35-37]{LN}.

Let $K$ be a field, and let $F(x) \in K[x]$ be a polynomial of
degree $s\geq 2$ with leading coefficient $a$. Then the {\it
Discriminant}, $\Disc(F)$, of $F(x)$ is defined by
\[
\Disc(F) = a^{2s-2} \prod_{i<j} (x_i - x_j)^2 \; ,
\]
where $x_0,x_1,\ldots,x_{s-1}$ are the roots of $F(x)$ in some
extension of $K$. Although $\Disc(F)$ is defined in terms of the
elements of an extension of $K$, it is actually an element of $K$
itself. There is an alternative formulation of $\Disc(F)$, given
below, which is very helpful for the computation of the
discriminant of a polynomial.

Let $G(x) \in K[x]$ and suppose $F(x)= a \prod_{i=0}^{s-1}(x-x_i)$
and $G(x) = b \prod_{j=0}^{t-1}(x-y_j)$, where $a,b \in K$ and
$x_0,x_1,\ldots,x_{s-1}$, $y_0,y_1,\ldots,y_{t-1}$ are in some
extension of $K$. Then the {\it Resultant}, $\Res(F,G)$, of $F(x)$
and $G(x)$ is

\begin{equation}\label{roots}
\Res(F,G) = (-1)^{st} b^s \prod_{j=0}^{t-1} F(y_j)
           = a^t \prod_{i=0}^{s-1} G(x_i) \; .
\end{equation}

The following statements are immediate from the definition of the
resultant of two polynomials.

\begin{cor}\label{cor1}
If $F$ is as above, and $G_1,G_2,R\in K[x]$, then
\begin{itemize}
\item[(i)] $\Res(F,-x)=F(0)$.

\item[(ii)] $\Res(F,G_1G_2)=\Res(F,G_1)\Res(F,G_2)$.

\end{itemize}
\end{cor}

\begin{cor} If $F$ is as above, and $F'\in K[x]$ is the
derivative of $F$, then

\begin{equation}\label{ff'}
{\Disc}(F)=(-1)^{\frac{s(s-1)}{2}}a^{s-l-2} \Res(F,F'),
\end{equation}
where $l$ is the degree of $F'$. Notice that if $K$ is of positive characteristic, then we may have $l<n-1$. 
\end{cor}

\begin{lemma}\label{Res-Disc}
	Let $u(x)$ and $v(x)\in K[x]$ be of degrees $m$ and $n$ and leading coefficients $a$ and $b$, respectively. Furthermore, suppose that the derivative of $u(x)$, i.e. $u'(x)$, is of degree $l$. Then
	\[
	\Res(u(v(x)),u'(v(x))=[(-1)^{\frac{m(m-1)}{2}}a^{-m+l+2}b^{ml}]^n\Disc(u(x))^n.
	\]
\end{lemma}

\begin{proof} 
	  Let $\alpha_1,\alpha_2,\ldots,\alpha_{m}$ be the roots of $u(x)=0$ in some extension of $K$.  Then the roots of $uv(x)$ are the collection of the roots of the equations $v(x)=\alpha_i$ for $i=1,2,\ldots,m$. We denote by $\beta_{ij}$, $j=1,2,\ldots,n$, the $n$ roots of $v(x)=\alpha_i$. So the roots of $uv(x)=0$ are $\beta_{ij}$ for $1\le i\le m$ and $1\le j\le n$. Since $u'(x)$ is of degree $l$, $u'(v(x))$ is of degree $nl$. Thus using~\eqref{roots} and the fact that $v(\beta_{ij})=\alpha_i$ for $1\le j\le n$ we have
	 \begin{eqnarray}
	 \Res(u(v(x)),u'(v(x))&=&(ab^{m})^{nl}\prod_{i=1}^{m}\prod_{j=1}^{n}u'(v(\beta_{ij}))=(ab^{m})^{nl}\prod_{i=1}^{m}u'(\alpha_i)^n\nonumber\\&=&(b^{m})^{nl}(a^l\prod_{i=1}^{m}u'(\alpha_i))^n=(b^{m})^{nl}[\Res(u(x),u'(x))]^n\nonumber\\&=&[(-1)^{\frac{m(m-1)}{2}}a^{-m+l+2}b^{ml}]^n\Disc(u(x))^n.
	 \end{eqnarray}
	 
\end{proof}

The following results are our main tools for
determining the parity of the number of irreducible factors of a
polynomial over a finite field.
\begin{theorem}
\cite{Pellet,Stickelberger} \label{thm-Stickel} Suppose that the
$s$-degree polynomial $f(x)\in \fq[x]$, where $q$ is an odd
prime power, is the product of $r$ pairwise distinct irreducible
polynomials over $\fq[x]$. Then $r\equiv s\pmod{2}$ if and only
if $\Disc(f)$ is a square in $\fq$.
\end{theorem}

\begin{theorem}
\cite{Dalen,Swan} \label{thm-Stickelberger} Suppose that
the $s$-degree polynomial $f(x)\in \ftwo[x]$ is the product of
$r$ pairwise distinct irreducible polynomials over $\ftwo[x]$ and,
let $F(x)\in \ZZ[x]$ be any monic lift of $f(x)$ to the
integers. Then $\Disc(F)\equiv 1$ or $5\pmod 8$, and more
importantly, $r\equiv s\pmod{2}$ if and only if $\Disc(F) \equiv 1 \pmod{8}$.
\end{theorem}

If $s$ is even and $\Disc(F)\equiv 1 \pmod{8}$, then
Theorem~\ref{thm-Stickelberger} asserts that $f(x)$ has an even
number of irreducible factors and therefore is reducible over
$\ftwo[x]$. Thus one can find necessary conditions for the
irreducibility of $f(x)$ by computing $\Disc(F)$ modulo 8.
\section{Stability of trinomials over finite  fields}

\subsection{Trinomials over the binary fields}

In this section we prove our main result about the instability of trinomials over the binary field $\ftwo$.  First we prove some results about the composition of arbitrary polynomials with trinomials.

\subsubsection{Composition of polynomials with trinomials}

\begin{lemma}\label{comp-trinomial}
 Let $f(x)=x^{2n}+x^{2n-s}+1,g(x)=x^{2m}+\sum_{i=0}^{2m-1}a_ix^i$ be two polynomials with integer coefficients such that both $s$ and $g(1)$ are odd numbers. Furthermore suppose that $a_{2m-1}$ is an even number  whenever $n=s$. Then 
	\[
\Res(gf(x), f'(x))\equiv 1 \pmod 8.
		\]
\end{lemma}
\begin{proof}
Let $R=\Res(gf(x), f'(x))$. We have 
\[
f'(x)=2nx^{2n-1}+(2n-s)x^{2n-s-1}=x^{2n-s-1}(2nx^s+2n-s). 
\]
Thus using Corollary~\ref{cor1}, we have
\begin{eqnarray}
R&=&\Res(gf(x),x^{2n-s-1})\Res(gf(x),2nx^s+2n-s)\\\nonumber&=&(gf(0))^{2n-s-1}\Res(gf(x),2nx^s+2n-s)\\\nonumber&=&(g(1))^{2n-s-1}\Res(gf(x),2nx^s+2n-s)\nonumber,
\end{eqnarray}
and hence since both $s$ and $g(1)$ are odd numbers we get
\begin{equation}\label{modular}
R=\Res(gf(x),2nx^s+2n-s) \pmod 8.
\end{equation}
Now let $\alpha_1,\alpha_2,\ldots,\alpha_{2m}$ be the roots of $g(x)=0$ in some extension of rational numbers.  Then the roots of $gf(x)$ are the union of the roots of the equations $f(x)=\alpha_i$ for $i=1,2,\ldots,2m$. We denote by $\beta_{ij}$, $j=1,2,\ldots,2n$, the $2n$ roots of $f(x)=\alpha_i$. So the roots of $gf(x)=0$ are $\beta_{ij}$ for $1\le i\le 2m$ and $1\le j\le 2n$. Using Newton identities for each $i$, $1\le i\le 2m$, we have: 
\begin{equation}\label{mono-s-power}
p_s(\beta_{i1},\beta_{i2},\ldots,\beta_{i(2n)})=\sum_{j=1}^{2n}\beta_{ij}^{s}=-s,
\end{equation}
\begin{equation}\label{2s-power}
p_{2s}(\beta_{i1},\beta_{i2},\cdots,\beta_{i(2n)})=\sum_{j=1}^{2n}\beta_{ij}^{2s}=\left\{
\begin{array}{cl}
s, & \;\;\;\mbox{if } s<n, \\
2n\alpha_i-n,          & \;\;\;\mbox{if } s=n, \\
s, & \;\;\;\mbox{if } n<s<2n, \\
\end{array}
\right.
\end{equation}
and hence
\begin{equation}\label{s-power}
p_s(\beta_{11},\beta_{12},\ldots,\beta_{(2m)(2n)})=\sum_{i=1}^{2m}p_s(\beta_{i1},\beta_{i2},\ldots,\beta_{i(2n)})=-2ms.
\end{equation}

From~\eqref{roots} and \eqref{modular} it follows that
\begin{eqnarray}\nonumber
R=\prod_{i=1}^{2m}\prod_{j=1}^{2n}(2n\beta_{ij}^s+2n-s)=(2n-s)^{4mn}+(2n-s)^{4mn-1}(2n)p_s(\beta_{11},\beta_{12},\ldots,\beta_{(2m)(2n)})\\+(2n-s)^{4mn-2}(2n)^2e_2(\beta_{11}^s,\beta_{12}^s,\ldots,\beta_{(2m)(2n)}^s)\nonumber\\+(2n-s)^{4mn-3}(2n)^3S(\beta_{11},\beta_{12},\ldots,\beta_{(2m)(2n)}),\nonumber
\end{eqnarray}
where $S(\beta_{11},\beta_{12},\ldots,\beta_{(2m)(2n)})$ is a symmetric polynomial in the roots of $gf(x)=0$ and hence it is an integer number. Thus using~\eqref{s-power} and the fact that $s$ is an odd number we get
\[
R=1-(2n-s)(4mn)s+(2n)^2e_2(\beta_{11}^s,\beta_{12}^s,\ldots,\beta_{(2m)(2n)}^s)\pmod 8.
\]

Now let $q_{si}=p_s(\beta_{i1},\beta_{i2},\cdots,\beta_{i(2n)})$ and $r_{is}=e_2(\beta_{i1}^s,\beta_{i2}^s,\cdots,\beta_{i(2n)}^s)$. Then
\[
e_2(\beta_{11}^s,\beta_{12}^s,\ldots,\beta_{(2m)(2n)}^s)=e_2(q_{s1},q_{s2},\cdots,q_{s(2m)})+\sum_{i=1}^{2m}e_2(\beta_{i1}^s,\beta_{i2}^s,\cdots,\beta_{i(2n)}^s).
\]
By~\eqref{mono-s-power} for each $i$ we have $q_{si}=-s$ which yields
\[
e_2(q_{s1},q_{s2},\cdots,q_{s(2m)})=\binom{2m}{2}s^2.
\]
Also we have
\[
 e_2(\beta_{i1}^s,\beta_{i2}^s,\cdots,\beta_{i(2n)}^s)=\frac{1}{2}[p_s(\beta_{i1},\beta_{i2},\cdots,\beta_{i(2n)})^2-p_{2s}(\beta_{i1},\beta_{i2},\cdots,\beta_{i(2n)})],
\]
and hence using~\eqref{2s-power}
\begin{equation}\label{e-sum}
e_2(\beta_{i1}^s,\beta_{i2}^s,\cdots,\beta_{i(2n)}^s)=\left\{
\begin{array}{cl}
\frac{1}{2}(s^2-s), & \;\;\;\mbox{if } s<n, \\
\frac{1}{2}(n^2-2n\alpha_i+n),          & \;\;\;\mbox{if } s=n, \\
\frac{1}{2}(s^2-s), & \;\;\;\mbox{if } n<s<2n. \\
\end{array}
\right.
\end{equation}

From the above equations we get that if $s\neq n$, then
\[
e_2(\beta_{11}^s,\beta_{12}^s,\ldots,\beta_{(2m)(2n)}^s)=\binom{2m}{2}s^2+(2m)\frac{1}{2}(s^2-s)=m(2m-1)s^2+m(s^2-s)
\]
 
from which we get that 
\[
R=1-(2n-s)(4mn)s+(2n)^2(m(2m-1)s^2+m(s^2-s))\pmod 8
\]
if $n\neq s$. It is easy to see that in this case $R\equiv 1\pmod 8$. Now let $n=s$. Then
\[
e_2(\beta_{11}^s,\beta_{12}^s,\ldots,\beta_{(2m)(2n)}^s)=\binom{2m}{2}n^2+\sum_{i=1}^{2m}\frac{1}{2}(n^2-2n\alpha_i+n)=\binom{2m}{2}n^2+m(n^2+n)-n\sum_{i=1}^{2m}\alpha_i.
\]

Thus 
\[
R\equiv 1-4mn^3+4n^2[m(2m-1)n^2+m(n^2+n)+na_{2m-1}]\pmod 8.
\]

Since we have assumed that $s=n$ is an odd number, we deduce that
\[
R\equiv 1+4a_{2m-1}\equiv 1 \pmod 8.
\]
\end{proof}

\begin{cor}\label{Disc-2n}
	 Let $f(x)=x^{2n}+x^{2n-s}+1,g(x)=x^{2m}+\sum_{i=0}^{2m-1}a_ix^i$ be two polynomials with integer coefficients such that both $s$ and $g(1)$ are odd numbers. Furthermore suppose that $a_{2m-1}$ is an even number  whenever $n=s$. Then 
	\[
	\Disc(g(f(x)))\equiv \Disc(g(x))^{2n} \pmod 8.
	\]
\end{cor}
\begin{proof}
Since degree of $gf(x)$ is divisible by 4, using Corollaries~\ref{cor1} and~\ref{ff'} we have
\[
\Disc(gf(x))=\Res(gf(x),f'(x)g'(f(x)))=\Res(gf(x),f'(x))\Res(gf(x),g'(f(x))),
\]	

and hence using Lemmata~\ref{comp-trinomial} and ~\ref{Res-Disc} we get
\[
\Disc(gf(x))=\Res(gf(x),g'(f(x)))=\Disc(g(x))^{2n}\pmod 8.
\]

\end{proof}	
\subsubsection{Instability of trinomials over $\ftwo$ }
The following is our main result about the trinomials over $\ftwo$.
\begin{theorem}\label{trim-even}
Let $f(x)=x^{2n}+x^{2n-s}+1$ be a trinomial over $\ftwo$. Then $f^{(3)}(x)$ is not an irreducible polynomial over $\ftwo$ and hence $f(x)$ is not stable.
\end{theorem}
\begin{proof}
If $s$ is an even number, then $f(x)=(x^n+x^{n-\frac{s}{2}}+1)^2$ which means that $f(x)$ is not irreducible and hence $f^{(3)}(x)$ is not an irreducible polynomial over $\ftwo$.  Now let $s$ be an odd number. If $n=1$, then $s=1$, $f(x)=x^2+x+1$, $ff(x)=x^4+x+1$, and $$f^{(3)}(x)=x^8+x^4+x^2+x+1= (x^4 + x^3 + 1)(x^4 + x^3 + x^2 + x + 1).$$ So let $n>1$. If $f(x)$ is not irreducible, then the claim is trivial. So assume that $f(x)$ is an irreducible polynomial and $F(x)=x^{2n}+x^{2n-s}+1$ is a lift of $f(x)$ to the integers. From Theorem~\ref{thm-Stickelberger} we deduce that $\Disc(F(x))\equiv 5 \pmod 8$. Now notice that $F(1)=3$ and since $n>1$, the coefficient of $x^{2n-1}$ is zero in $F(x)$ whenever $n=s$. Thus using Corollary~\ref{Disc-2n} we get $\Disc(FF(x))\equiv 1\pmod 8$ which in turn using Theorem~\ref{thm-Stickelberger} implies that $ff(x)$ is not an irreducible polynomial over $\ftwo$. 	
	
\end{proof}	
\begin{remark}
The conclusion of Theorem~\ref{trim-even} does not hold for trinomials of odd degree. For example, if $f(x)=x^3+x^2+1$, our limited computations with MAGMA computer algebra package shows that $f^{(n)}(x)$ is irreducible for $n\le 10$, and even more, it is primitive for $n\le 6$. Our guess is that it is stable over $\ftwo$. 
\end{remark}

\begin{remark}
	It is possible to prove results similar to the results of this section for monic polynomials of even degree over finite fields of characteristic two. 
\end{remark}

\subsection{Polynomials over fields of odd characteristic}
In this section we consider trinomials over the finite fields of odd characteristic and show that some families of them are not stable.  

\begin{theorem}
	Let  $p>2$ be a prime number such that $p\mid n$,
	and let $f(x)=x^{2n}+ax^{2s+1}+b$ be a trinomial over $\FF_{p^t}$ and $g(x)$ be a monic polynomial of $\deg(g)=2m$ in $\FF_{p^t}[x]$.
	Then $gf(x)$ is reducible. In particular $f(x)$ is not stable.
\end{theorem}

\begin{proof}
	It suffices to show that the discriminant of $gf(x)$ is a quadratic residue in $\FF_{p^t}$. Since then by Theorem~\ref{thm-Stickel}, $gf(x)$ has an even number of factors in $\FF_{p^t}[x]$ and hence it is reducible. We have 
	$f'(x)=a(2s+1)x^{2s}$. Since the degree of $gf(x)$ is divisible by $4$, from~\eqref{ff'} we get
	$$
	\Disc(gf(x))=\Res(gf(x),f'(x)g'(f(x)))=\Res( gf(x),f'(x)) \Res(gf(x),g'(f(x))).
	$$
	Now on the one hand using Corollary~\ref{cor1} we have
	$$
	\Res( gf(x),f'(x))= (a(2s+1))^{4nm}g(f(0))^{2s}=(a(2s+1))^{4mn}g(b)^{2s}
	$$
	which shows that $\Res(gf(x),g'(f(x)))$ is a quadratic residue, and on the other hand from Lemma~\ref{Res-Disc} we have
	$$
	\Res(gf(x),g'(f(x)))=\Disc(g(x))^{2n}.
	$$
 Thus from the equations above we conclude that $\Disc(gf(x))$ is a quadratic residue in $\FF_{p^t}$ which finishes the proof.
\end{proof}

In the theorem above, we dealt with the polynomials of even degree. It is possible to apply Theorem~\ref{thm-Stickelberger} to prove instability of some families of trinomials of odd degree. For example we have the following theorem.
\begin{theorem}
 Let $p\not\equiv 1\pmod 8$, and let $f(x)=x^{p}+ax^2+b$ be a polynomial over $\FF_p$. Furthermore suppose that $ab=-3$.
 Then $f(x)$ is not stable over $\FF_p$.
\end{theorem}
\begin{proof}
From Corollary~\ref{ff'}, we have $\Disc(f(x))=-b(2a)^{p}=-b(2a)=6$ and thus using Lemma \ref{Res-Disc} we get
\begin{eqnarray} 
 \Disc(f(f(x)))&=&\Disc(f)^{p}\Res(ff,f')=6\Res(ff,f')
  =6\Res(ff,f')\nonumber\\&=&-12aff(0)=-12a(b^{p}+ab^2+b)=-12a(ab^2+2b)=-36. \nonumber
\end{eqnarray}  
Now if $p\equiv 3,7\pmod 8$, then $\Disc(ff(x))=-36$ is a quadratic non-residue and hence from Theorem~\ref{thm-Stickel} it follows that $ff(x)$ is not irreducible. So in order to finish the proof we need to prove the claim for $p\equiv 5\pmod 8$. Now we have 
 \[
 \Disc(fff(x))=\Disc(ff(x))^p \Res(fff,f')
 \] 
and hence
\begin{align*}
\Res(fff,f')=-2afff(0)&=-2a((ab^2+2b)+a(ab^2+2b)^2+b)\\
&=-2ab(ab+3+ab(ab+2)^2)=-18.
\end{align*}
But as $-2$ is a quadratic non-residue when $p\equiv 5\pmod 8$, it follows that $\Disc(fff(x))=-18$ is a quadratic non residue and thus $fff(x)$ is not irreducible by Theorem \ref{thm-Stickel}.
\end{proof}

\section{Polynomials of higher weights}

It is possible to prove results similar to the results of the previous section for some families of polynomials which are of higher weight, i.e., polynomials which have more than three nonzero terms. As such are the following theorems about some families of polynomials over the binary field. The proofs of the following theorems are very similar to the proofs of theorems of the previous section and hence we omit them.
\begin{theorem}\label{higher-weight}
	Let $f(x)=x^{2n}+x^s+g(x^8)+1$ be a polynomial over $\ftwo$ such that $8\deg(g)<s$. Then $f(x)$ is not stable over $\ftwo$. 
\end{theorem}

\begin{theorem}
	Let $f(x)=g(x^4)+x^s+1$ be a polynomial over $\ftwo$ such that $s<4\deg(g)$. Then $f(x)$ is not stable over $\ftwo$. 
\end{theorem}

\section{Concluding remarks}
As we noted in the introduction in~\cite{AOLS} the stability of quadratic polynomials over binary fields were studied. In~\cite{DNOS} the stability of the polynomials of degree three over fields of characteristic three has been studied. It seems that methods of~\cite{AOLS, DNOS} are very hard to apply for polynomials of higher degree or higher weight. In this paper, we used Theorems~\ref{thm-Stickel} and ~\ref{thm-Stickelberger} to study the stability of trinomials and some families of polynomials of higher weight over finite fields. This method also seems to be not that much of help for attacking our conjecture for polynomials of higher weight. For example, if  $f(x)=x^{20}+x^{18}+x^5+x^2+1$, then
computations with MAGMA computer algebra package shows that if $F(x)$ is a monic lift of $f(x)$ over the integers, then $\Disc(F(x))=\Disc(FF(x))=\Disc(FFF(x))\equiv 5 \pmod 8$ while $f(x)$ and $ff(x)$ are irreducible over $\ftwo$ and $fff(x)$ is not  irreducible over $\ftwo$. Thus one cannot hope to use Theorems~\ref{thm-Stickel} and~\ref{thm-Stickelberger} solely to resolve our conjecture. 

 We also noted in the introduction that Odoni~\cite{Odoni} studied the stability of additive polynomial $f(x)=x^p-x-1$ over $\fp$. His method is different from the methods of the current paper and those of~\cite{AOLS, DNOS}. Using Capelli's lemma he showed that the Galois group of $ff(x)$ over $\fp$ is the cyclic group of order $p$ and hence $ff(x)$ is not stable. It is not clear how to generalize Odoni's method to the case of polynomials which are not additive over $\fp$. In conclusion, it seems that new ideas and methods are needed to be able to confirm Conjecture~\ref{main-conj} if it is correct.


\begin{bibdiv}
\begin{biblist}	
	
\bib{AOLS}{article}{
	author={Ahmadi, Omran},
	author={Luca, Florian},
	author={Ostafe, Alina},
	author={Shparlinski, Igor E.},
	title={On stable quadratic polynomials},
	journal={Glasg. Math. J.},
	volume={54},
	date={2012},
	number={2},
	pages={359--369},
	issn={0017-0895},
	review={\MR{2911375}},
	doi={10.1017/S001708951200002X},
}	

\bib{Nidal}{article}{
	author={Ali, Nidal},
	title={Stabilit\'e des polyn\^omes},
	language={French},
	journal={Acta Arith.},
	volume={119},
	date={2005},
	number={1},
	pages={53--63},
	issn={0065-1036},
	review={\MR{2163517}},
	doi={10.4064/aa119-1-4},
}

\bib{Ayad}{article}{
	author={Ayad, Mohamed},
	author={McQuillan, Donald L.},
	title={Corrections to: ``Irreducibility of the iterates of a quadratic
		polynomial over a field'' [Acta Arith. {\bf 93} (2000), no. 1, 87--97;
		MR1760091 (2001c:11031)]},
	journal={Acta Arith.},
	volume={99},
	date={2001},
	number={1},
	pages={97},
	issn={0065-1036},
	review={\MR{1845367}},
	doi={10.4064/aa99-1-9},
}
\bib{Cohen-1}{article}{
	author={Cohen, Stephen D.},
	title={On irreducible polynomials of certain types in finite fields},
	journal={Proc. Cambridge Philos. Soc.},
	volume={66},
	date={1969},
	pages={335--344},
	review={\MR{0244202}},
}

\bib{Dalen}{article}{
	author={Dalen, K\aa re},
	title={On a theorem of Stickelberger},
	journal={Math. Scand.},
	volume={3},
	date={1955},
	pages={124--126},
	issn={0025-5521},
	review={\MR{0071460}},
	doi={10.7146/math.scand.a-10433},
}


\bib{Domingo}{article}{
	author={G\'omez-P\'erez, Domingo},
	title={Personal communication},

}	
	
	

\bib{DNOS}{article}{
	author={G\'omez-P\'erez, Domingo},
	author={Nicol\'as, Alejandro P.},
	author={Ostafe, Alina},
	author={Sadornil, Daniel},
	title={Stable polynomials over finite fields},
	journal={Rev. Mat. Iberoam.},
	volume={30},
	date={2014},
	number={2},
	pages={523--535},
	issn={0213-2230},
	review={\MR{3231208}},
	doi={10.4171/RMI/791},
}

\bib{Goksel}{article}{
	author={Goksel, Vefa},
	author={Xia, Shixiang},
	author={Boston, Nigel},
	title={A refined conjecture for factorizations of iterates of quadratic
		polynomials over finite fields},
	journal={Exp. Math.},
	volume={24},
	date={2015},
	number={3},
	pages={304--311},
	issn={1058-6458},
	review={\MR{3359218}},
	doi={10.1080/10586458.2014.992079},
}



\bib{Jones-Boston}{article}{
	author={Jones, Rafe},
	author={Boston, Nigel},
	title={Settled polynomials over finite fields},
	journal={Proc. Amer. Math. Soc.},
	volume={140},
	date={2012},
	number={6},
	pages={1849--1863},
	issn={0002-9939},
	review={\MR{2888174}},
	doi={10.1090/S0002-9939-2011-11054-2},
}

\bib{LN}{book}{
	author={Lidl, Rudolf},
	author={Niederreiter},
	title={Finite Fields},
    issn={Cambridge University Press},
    date={1984},
}

\bib{Odoni}{article}{
	author={Odoni, R. W. K.},
	title={The Galois theory of iterates and composites of polynomials},
	journal={Proc. London Math. Soc. (3)},
	volume={51},
	date={1985},
	number={3},
	pages={385--414},
	issn={0024-6115},
	review={\MR{805714}},
	doi={10.1112/plms/s3-51.3.385},
}


\bib{Pellet}{article}{
	
	author={Pellet, A.-E.},
	title={Sur la décomposition d'une fonction entière en facteurs irréductibles suivant un module premier p.},
	journal={Comptes Rendus de l'Académie des Sciences Paris},
	volume={86},
	date={1878},
	pages={1071--1072},

}


\bib{Stickelberger}{article}{
	author={Stickelberger, Ludwig},
    title={ Über eine neue Eigenschaft der Diskriminanten algebraischer Zahlkörper.},
    journal={Verhandlungen des ersten Internationalen Mathematiker-Kongresses, Zürich},
    volume={},
    date={1897},
   pages={182--193},

}


\bib{Swan}{article}{
	author={Swan, Richard G.},
	title={Factorization of polynomials over finite fields},
	journal={Pacific J. Math.},
	volume={12},
	date={1962},
	pages={1099--1106},
	issn={0030-8730},
	review={\MR{0144891}},
}

\end{biblist}

\end{bibdiv}
\end{document}